\documentclass[reqno, 12pt]{amsart}
\pdfoutput=1
\makeatletter
\let\origsection=\section \def\section{\@ifstar{\origsection*}{\mysection}}
\def\mysection{\@startsection{section}{1}\z@{.7\linespacing\@plus\linespacing}{.5\linespacing}{\normalfont\scshape\centering\S}}
\makeatother

\usepackage{amsmath,amssymb,amsthm}
\usepackage{mathrsfs}
\usepackage{mathabx}\changenotsign
\usepackage{dsfont}
\usepackage{bbm}

\usepackage{xcolor}
\usepackage[backref=section]{hyperref}
\usepackage[ocgcolorlinks]{ocgx2}
\hypersetup{
	colorlinks=true,
	linkcolor={red!60!black},
	citecolor={green!60!black},
	urlcolor={blue!60!black},
}

\usepackage[open,openlevel=2,atend]{bookmark}

\usepackage[abbrev,msc-links,backrefs]{amsrefs}
\usepackage{doi}

\renewcommand{\PrintDOI}[1]{\doi{#1}}

\usepackage[T1]{fontenc}
\usepackage{lmodern}
\usepackage[babel]{microtype}
\usepackage[english]{babel}

\usepackage{geometry}
\numberwithin{equation}{section}
\numberwithin{figure}{section}

\usepackage{enumitem}

\let\polishlcross=\l
\def\l{\ifmmode\ell\else\polishlcross\fi}

\let\emptyset=\varnothing
\let\setminus=\smallsetminus

\makeatletter
\def\moverlay{\mathpalette\mov@rlay}
\def\mov@rlay#1#2{\leavevmode\vtop{   \baselineskip\z@skip \lineskiplimit-\maxdimen
		\ialign{\hfil$\m@th#1##$\hfil\cr#2\crcr}}}
\newcommand{\charfusion}[3][\mathord]{
	#1{\ifx#1\mathop\vphantom{#2}\fi
		\mathpalette\mov@rlay{#2\cr#3}
	}
	\ifx#1\mathop\expandafter\displaylimits\fi}
\makeatother

\newcommand{\dcup}{\charfusion[\mathbin]{\cup}{\cdot}}

\DeclareFontFamily{U}  {MnSymbolC}{}
\DeclareSymbolFont{MnSyC}         {U}  {MnSymbolC}{m}{n}
\DeclareFontShape{U}{MnSymbolC}{m}{n}{
	<-6>  MnSymbolC5
	<6-7>  MnSymbolC6
	<7-8>  MnSymbolC7
	<8-9>  MnSymbolC8
	<9-10> MnSymbolC9
	<10-12> MnSymbolC10
	<12->   MnSymbolC12}{}
\DeclareMathSymbol{\powerset}{\mathord}{MnSyC}{180}

\usepackage{tikz}
\usetikzlibrary{calc,decorations.pathmorphing}
\usetikzlibrary{arrows,decorations.pathreplacing}
\pgfdeclarelayer{background}
\pgfdeclarelayer{foreground}
\pgfdeclarelayer{front}
\pgfsetlayers{background,main,foreground,front}

\usepackage{multicol}
\usepackage{subcaption}
\newcommand{\pedge}[9]{
	
	\ifx\relax#6\relax
	\def\qoffs{0pt}
	\else
	\def\qoffs{#6}
	\fi
	
	\def\phedge{
		($#1+#5!\qoffs!-90:#2-#5$) -- 
		($#2+#1!\qoffs!-90:#3-#1$) -- 
		($#3+#2!\qoffs!-90:#4-#2$) -- 
		($#4+#3!\qoffs!-90:#5-#3$) -- 
		($#5+#4!\qoffs!-90:#1-#4$) -- cycle}

	\coordinate (12) at ($#1!\qoffs!90:#2$);
	\coordinate (15) at ($#1!\qoffs!-90:#5$);
	\coordinate (23) at ($#2!\qoffs!90:#3$);
	\coordinate (21) at ($#2!\qoffs!-90:#1$);
	\coordinate (34) at ($#3!\qoffs!90:#4$);
	\coordinate (32) at ($#3!\qoffs!-90:#2$);
	\coordinate (45) at ($#4!\qoffs!90:#5$);
	\coordinate (43) at ($#4!\qoffs!-90:#3$);
	\coordinate (51) at ($#5!\qoffs!90:#1$);
	\coordinate (54) at ($#5!\qoffs!-90:#4$);

	\def\nphedge{
		(15) let \p1=($(15)-#1$), \p2=($(12)-#1$) in 
		arc[start angle={atan2(\y1,\x1)}, delta angle={atan2(\y2,\x2)-atan2(\y1,\x1)-360*(atan2(\y2,\x2)-atan2(\y1,\x1)>0)}, x radius=\qoffs, y radius=\qoffs] --
		(21) let \p1=($(21)-#2$), \p2=($(23)-#2$) in 
		arc[start angle={atan2(\y1,\x1)}, delta angle={atan2(\y2,\x2)-atan2(\y1,\x1)-360*(atan2(\y2,\x2)-atan2(\y1,\x1)>0)}, x radius=\qoffs, y radius=\qoffs] --
		(32) let \p1=($(32)-#3$), \p2=($(34)-#3$) in 
		arc[start angle={atan2(\y1,\x1)}, delta angle={atan2(\y2,\x2)-atan2(\y1,\x1)-360*(atan2(\y2,\x2)-atan2(\y1,\x1)>0)}, x radius=\qoffs, y radius=\qoffs] --
		(43) let \p1=($(43)-#4$), \p2=($(45)-#4$) in 
		arc[start angle={atan2(\y1,\x1)}, delta angle={atan2(\y2,\x2)-atan2(\y1,\x1)-360*(atan2(\y2,\x2)-atan2(\y1,\x1)>0)}, x radius=\qoffs, y radius=\qoffs] --
		(54) let \p1=($(54)-#5$), \p2=($(51)-#5$) in 
		arc[start angle={atan2(\y1,\x1)}, delta angle={atan2(\y2,\x2)-atan2(\y1,\x1)-360*(atan2(\y2,\x2)-atan2(\y1,\x1)>0)}, x radius=\qoffs, y radius=\qoffs] --
		cycle}

	\ifx\relax#7\relax
	\def\plwidth{1pt}
	\else
	\def\plwidth{#7}
	\fi
	
	\ifx\relax#9\relax
	\fill \nphedge;
	\else
	\fill[#9]\nphedge;
	\fi
	
	\ifx\relax#8\relax
	\draw[line width=\plwidth,rounded corners=\qoffs]\nphedge;
	\else
	\draw[line width=\plwidth,#8]\nphedge;
	\fi
}

\newcommand{\qedge}[7]{
	
	\ifx\relax#4\relax
	\def\qoffs{0pt}
	\else
	\def\qoffs{#4}
	\fi
	
	\def\qhedge{
		($#1+#3!\qoffs!-90:#2-#3$) --
		($#2+#1!\qoffs!-90:#3-#1$) --
		($#3+#2!\qoffs!-90:#1-#2$) -- cycle}

	\coordinate (12) at ($#1!\qoffs!90:#2$);
	\coordinate (13) at ($#1!\qoffs!-90:#3$);
	\coordinate (23) at ($#2!\qoffs!90:#3$);
	\coordinate (21) at ($#2!\qoffs!-90:#1$);
	\coordinate (31) at ($#3!\qoffs!90:#1$);
	\coordinate (32) at ($#3!\qoffs!-90:#2$);
	
	\def\nqhedge{
		(13) let \p1=($(13)-#1$), \p2=($(12)-#1$) in
		arc[start angle={atan2(\y1,\x1)}, delta angle={atan2(\y2,\x2)-atan2(\y1,\x1)-360*(atan2(\y2,\x2)-atan2(\y1,\x1)>0)}, x radius=\qoffs, y radius=\qoffs] --
		(21) let \p1=($(21)-#2$), \p2=($(23)-#2$) in
		arc[start angle={atan2(\y1,\x1)}, delta angle={atan2(\y2,\x2)-atan2(\y1,\x1)-360*(atan2(\y2,\x2)-atan2(\y1,\x1)>0)}, x radius=\qoffs, y radius=\qoffs] --
		(32) let \p1=($(32)-#3$), \p2=($(31)-#3$) in
		arc[start angle={atan2(\y1,\x1)}, delta angle={atan2(\y2,\x2)-atan2(\y1,\x1)-360*(atan2(\y2,\x2)-atan2(\y1,\x1)>0)}, x radius=\qoffs, y radius=\qoffs] --
		cycle}
	
	\ifx\relax#5\relax
	\def\qlwidth{1pt}
	\else
	\def\qlwidth{#5}
	\fi
	
	\ifx\relax#7\relax
	\fill \nqhedge;
	\else
	\fill[#7]\nqhedge;
	\fi
	
	\ifx\relax#6\relax
	\draw[line width=\qlwidth,rounded corners=\qoffs]\nqhedge;
	\else
	\draw[line width=\qlwidth,#6]\nqhedge;
	\fi
}

\newcommand{\redge}[8]{
	
	\ifx\relax#5\relax
	\def\qoffs{0pt}
	\else
	\def\qoffs{#5}
	\fi
	
	\def\rhedge{
		($#1+#4!\qoffs!-90:#2-#4$) -- 
		($#2+#1!\qoffs!-90:#3-#1$) -- 
		($#3+#2!\qoffs!-90:#4-#2$) -- 
		($#4+#3!\qoffs!-90:#1-#3$) -- cycle}

	\coordinate (12) at ($#1!\qoffs!90:#2$);
	\coordinate (14) at ($#1!\qoffs!-90:#4$);
	\coordinate (23) at ($#2!\qoffs!90:#3$);
	\coordinate (21) at ($#2!\qoffs!-90:#1$);
	\coordinate (34) at ($#3!\qoffs!90:#4$);
	\coordinate (32) at ($#3!\qoffs!-90:#2$);
	\coordinate (41) at ($#4!\qoffs!90:#1$);
	\coordinate (43) at ($#4!\qoffs!-90:#3$);
	
	\def\nrhedge{
		(14) let \p1=($(14)-#1$), \p2=($(12)-#1$) in 
		arc[start angle={atan2(\y1,\x1)}, delta angle={atan2(\y2,\x2)-atan2(\y1,\x1)-360*(atan2(\y2,\x2)-atan2(\y1,\x1)>0)}, x radius=\qoffs, y radius=\qoffs] --
		(21) let \p1=($(21)-#2$), \p2=($(23)-#2$) in 
		arc[start angle={atan2(\y1,\x1)}, delta angle={atan2(\y2,\x2)-atan2(\y1,\x1)-360*(atan2(\y2,\x2)-atan2(\y1,\x1)>0)}, x radius=\qoffs, y radius=\qoffs] --
		(32) let \p1=($(32)-#3$), \p2=($(34)-#3$) in 
		arc[start angle={atan2(\y1,\x1)}, delta angle={atan2(\y2,\x2)-atan2(\y1,\x1)-360*(atan2(\y2,\x2)-atan2(\y1,\x1)>0)}, x radius=\qoffs, y radius=\qoffs] --
		(43) let \p1=($(43)-#4$), \p2=($(41)-#4$) in 
		arc[start angle={atan2(\y1,\x1)}, delta angle={atan2(\y2,\x2)-atan2(\y1,\x1)-360*(atan2(\y2,\x2)-atan2(\y1,\x1)>0)}, x radius=\qoffs, y radius=\qoffs] --
		cycle}
	
	\ifx\relax#6\relax
	\def\rlwidth{1pt}
	\else
	\def\rlwidth{#6}
	\fi
	
	\ifx\relax#8\relax
	\fill \nrhedge;
	\else
	\fill[#8]\nrhedge;
	\fi
	
	\ifx\relax#7\relax
	\draw[line width=\rlwidth,rounded corners=\qoffs]\nrhedge;
	\else
	\draw[line width=\rlwidth,#7]\nrhedge;
	\fi
}

\let\epsilon=\varepsilon

\let\rho=\varrho
\let\theta=\vartheta

\newcommand{\cF}{\mathcal{F}}
\newcommand{\cG}{\mathcal{G}}
\newcommand{\cH}{\mathcal{H}}

\newcommand{\cP}{\mathcal{P}}

\newcommand{\ccP}{\mathscr{P}}

\newtheoremstyle{note}  {4pt}  {4pt}  {\sl}  {}  {\bfseries}  {.}  {.5em}          {}
\newtheoremstyle{introthms}  {3pt}  {3pt}  {\itshape}  {}  {\bfseries}  {.}  {.5em}          {\thmnote{#3}}
\newtheoremstyle{remark}  {2pt}  {2pt}  {\rm}  {}  {\bfseries}  {.}  {.3em}          {}

\theoremstyle{plain}
\newtheorem{theorem}{Theorem}[section]
\newtheorem{lemma}[theorem]{Lemma}

\newtheorem{conj}[theorem]{Conjecture}

\newtheorem{claim}[theorem]{Claim}

\theoremstyle{note}
\newtheorem{dfn}[theorem]{Definition}

\theoremstyle{remark}

\usepackage{lineno}
\newcommand*\patchAmsMathEnvironmentForLineno[1]{
	\expandafter\let\csname old#1\expandafter\endcsname\csname #1\endcsname
	\expandafter\let\csname oldend#1\expandafter\endcsname\csname end#1\endcsname
	\renewenvironment{#1}
	{\linenomath\csname old#1\endcsname}
	{\csname oldend#1\endcsname\endlinenomath}}
\newcommand*\patchBothAmsMathEnvironmentsForLineno[1]{
	\patchAmsMathEnvironmentForLineno{#1}
	\patchAmsMathEnvironmentForLineno{#1*}}
\AtBeginDocument{
	\patchBothAmsMathEnvironmentsForLineno{equation}
	\patchBothAmsMathEnvironmentsForLineno{align}
	\patchBothAmsMathEnvironmentsForLineno{flalign}
	\patchBothAmsMathEnvironmentsForLineno{alignat}
	\patchBothAmsMathEnvironmentsForLineno{gather}
	\patchBothAmsMathEnvironmentsForLineno{multline}
}

\usepackage{scalerel}

\makeatletter
\newcommand{\overrighharpoonup}[1]{\ThisStyle{%
		\vbox {\m@th\ialign{##\crcr
				\rightharpoonupfill \crcr
				\noalign{\kern-\p@\nointerlineskip}
				$\hfil\SavedStyle#1\hfil$\crcr}}}}

\def\rightharpoonupfill{%
	$\SavedStyle\m@th\mkern+0.8mu\cleaders\hbox{$\shortbar\mkern-4mu$}\hfill\rightharpoonuptip\mkern+0.8mu$}

\def\rightharpoonuptip{%
	\raisebox{\z@}[2pt][1pt]{\scalebox{0.55}{$\SavedStyle\rightharpoonup$}}}

\def\shortbar{%
	\smash{\scalebox{0.55}{$\SavedStyle\relbar$}}}
\makeatother

\makeatletter
\newcommand{\overlefharpoonup}[1]{\ThisStyle{%
		\vbox {\m@th\ialign{##\crcr
				\leftharpoonupfill \crcr
				\noalign{\kern-\p@\nointerlineskip}
				$\hfil\SavedStyle#1\hfil$\crcr}}}}

\def\leftharpoonupfill{%
	$\SavedStyle\m@th\mkern+0.8mu\cleaders\hbox{$\shortbar\mkern-4mu$}\hfill\leftharpoonuptip\mkern+0.8mu$}

\def\leftharpoonuptip{%
	\raisebox{\z@}[2pt][1pt]{\scalebox{0.55}{$\SavedStyle\leftharpoonup$}}}

\makeatletter
\newsavebox\myboxA
\newsavebox\myboxB
\newlength\mylenA

\newcommand*\xoverline[2][0.75]{%
	\sbox{\myboxA}{$\m@th#2$}%
	\setbox\myboxB\null% Phantom box
	\ht\myboxB=\ht\myboxA%
	\dp\myboxB=\dp\myboxA%
	\wd\myboxB=#1\wd\myboxA% Scale phantom
	\sbox\myboxB{$\m@th\overline{\copy\myboxB}$}%  Overlined phantom
	\setlength\mylenA{\the\wd\myboxA}%   calc width diff
	\addtolength\mylenA{-\the\wd\myboxB}%
	\ifdim\wd\myboxB<\wd\myboxA%
	\rlap{\hskip 0.5\mylenA\usebox\myboxB}{\usebox\myboxA}%
	\else
	\hskip -0.5\mylenA\rlap{\usebox\myboxA}{\hskip 0.5\mylenA\usebox\myboxB}%
	\fi}
\makeatother

\begin{document}
	
	\title[A local version of Katona's intersection theorem]
	{A local version of Katona's intersection theorem}
	
	\author[M. Sales]{Marcelo Sales}
	\address{Department of Mathematics, Emory University, Atlanta, USA}
	\email{marcelo.tadeu.sales@emory.edu}
	
	\author[B. Sch\"ulke]{Bjarne Sch\"ulke}
	\address{Department of Mathematics, California Institute of Technology, Pasadena, USA}
	\email{schuelke@caltech.edu}
	
	\subjclass[2010]{}
	
	\begin{abstract}
		Katona's intersection theorem states that every intersecting family~$\cF\subseteq[n]^{(k)}$ satisfies~$\vert\partial\cF\vert\geq\vert\cF\vert$, where~$\partial\cF=\{F\setminus x:x\in F\in\cF\}$ is the shadow of~$\cF$.
		Frankl conjectured that for~$n>2k$ and every intersecting family~$\cF\subseteq [n]^{(k)}$, there is some~$i\in[n]$ such that~$\vert \partial \cF(i)\vert\geq \vert\cF(i)\vert$, where~$\cF(i)=\{F\setminus i:i\in F\in\cF\}$ is the link of~$\cF$ at~$i$.
		Here, we prove this conjecture in a very strong form for~$n> \binom{k+1}{2}$.
		In particular, our result implies that for any~$j\in[k]$, there is a $j$-set~$\{a_1,\dots,a_j\}\in[n]^{(j)}$ such that~$\vert \partial \cF(a_1,\dots,a_j)\vert\geq \vert\cF(a_1,\dots,a_j)\vert$. A similar statement is also obtained for cross-intersecting families.
		
		%Further, Frankl conjectured that for~$n> k+\ell$ and cross-intersecting families~$\cG\subseteq [n]^{(\ell)}$ and~$\cH\subseteq [n]^{(k)}$, there is some~$i\in[n]$ such that~$\vert\partial\cG(i)\vert\geq\vert\cG(i)\vert$ or~$\vert\partial\cH(i)\vert\geq\vert\cH(i)\vert$.
		%We prove this conjecture again in a very strong form for~$n> k\ell$.
	\end{abstract}
	
	\maketitle
	
	\section{Introduction}\label{sec:intro}
		Throughout the paper, let~$n,k,\ell$ be positive integers. Let~$[n]=\{1,\ldots,n\}$ and for a set~$X$ let $X^{(k)}=\{A\subseteq X:\: |A|=k\}$ be the set of $k$-subsets of $X$. 
		A family~$\cF\subseteq[n]^{(k)}$ is called \emph{intersecting} if~$F\cap F'\neq\emptyset$ for all~$F,F'\in\cF$ and the \emph{shadow} of~$\cF$ is $$\partial\cF=\{F\setminus x:x\in F\in\cF\}\,.$$
		%Intersecting families and the shadow of a family are amongst the most prominent objects in extremal set theory.
		Extremal properties of shadows and intersecting families are amongst the most prominent topics in extremal set theory.
		%For instance, the Erd\H{o}s-Ko-Rado theorem determines the maximum size of an intersecting ($k$-uniform) family and the Kruskal-Katona theorem provides a lower bound on the shadow of a~$k$-uniform family.
		For instance, two cornerstones of the area are the Erd\H{o}s--Ko--Rado theorem~\cite{EKR:61}, which determines the maximum size of an intersecting family and the Kruskal--Katona theorem~\cites{Kr:63, Ka:68}, which provides a solution for the minimisation problem of the shadow.
		
		The following celebrated theorem due to Katona~\cite{K:64} combines these two concepts by bounding the size of the shadow of an intersecting family.
	
		\begin{theorem}[\cite{K:64}]\label{thm:Katona}
			Suppose~$\cF\subseteq[n]^{(k)}$ is intersecting.
			Then~$\vert\partial\cF\vert\geq\vert\cF\vert$.
		\end{theorem}
	
		%In fact, Katona~\cite{K:64} also provided a bound for~$t$-intersecting families and recently Frankl and Katona~\cite{FK:21} and Liu and Mubayi~\cite{LM:21} extended the original result.
		Theorem \ref{thm:Katona} was proved in a more general setting in \cite{K:64}. Recently, the result was improved by Frankl and Katona~\cite{FK:21} and Liu and Mubayi~\cite{LM:21} for intersecting families of larger size.
		%In addition, Frankl~\cite{F:76} proved the respective result for cross-intersecting families, see Theorem~\ref{thm:Franklcrossint}.
		It is also worth to note the cross-intersecting variant of Theorem \ref{thm:Katona} that Frankl proved in~\cite{F:76}. Given integers~$k,\ell\geq 1$, a pair of families~$\cF\subseteq [n]^{(k)}$, $\cG\subseteq [n]^{(\ell)}$ is \textit{cross-intersecting} if for every~$F\in \cF$ and~$G\in \cG$ we have~$F\cap G\neq \emptyset$.
		
		\begin{theorem}[\cite{F:76}]\label{thm:Franklcrossint}
			Let~$1\leq k,\ell \leq n$ be positive integers and $\cF\subseteq[n]^{(k)}$ and~$\cG\subseteq [n]^{(\ell)}$ be cross-intersecting families.
			Then either~$\vert\partial\cF\vert\geq\vert\cF\vert$ or~$\vert\partial\cG\vert\geq\vert\cG\vert$.
		\end{theorem}

		%Frankl~\cite{F:21} made the following conjecture regarding a local version of Theorem~\ref{thm:Katona}.
		%\begin{conj}[Frankl]\label{conj:Frankl}
		%	Let~$n>2k$ and let~$\cF\subseteq[n]^{(k)}$ be intersecting.
		%	Then there is some~$i\in[n]$ such that~$\vert\partial\cF(i)\vert\geq\vert\cF(i)\vert$.
		%\end{conj}
		%Here~$\cF(i)=\{F\setminus i:i\in F\in\cF\}$ denotes the \emph{link} of~$\cF$ at~$i$ (as usual, we omit the parentheses around singletons).
		%Note that while the shifting technique can be used directly to prove Theorem~\ref{thm:Katona}, this is not the case for this conjecture.
		%Here we prove Frankl's conjecture for~$n> \frac{(k+1)k}{2}$.

		In this paper we establish a local version of Theorem~\ref{thm:Katona}.
		For a family~$\cF\subseteq [n]^{(k)}$ and sets~$A, B\subseteq [n]$, we define
		\begin{align*}
			\cF(A,\overline{B})=\{F\setminus A:\: F\in \cF \text{ such that }A\subseteq F \text{ and }B\cap F=\emptyset\}
		\end{align*}
		as the link of~$\cF$ at~$A$ induced on~$[n]\setminus B$.
		Notice that for~$B=\emptyset$, the family~$\cF(A):=\cF(A,\emptyset)$ is just the usual link of~$\cF$ at~$A$.
		If on the other hand~$A=\emptyset$, the family~$\cF(\overline{B}):=\cF(\emptyset, \overline{B})$ is just the induced hypergraph on the set~$[n]\setminus B$.
		Observe that~$\partial(\cF(A))=(\partial\cF)(A)$ and~$\partial(\cF(A,\overline{B}))\subseteq(\partial\cF)(A,\overline{B})$.
		We write~$\partial\cF(A,\overline{B}):=\partial(\cF(A,\overline{B}))$.

		Frankl~\cite{F:21} conjectured the following local version of Theorem~\ref{thm:Katona}: Let~$n>2k$ and let~$\cF \subseteq [n]^{(k)}$ be an intersecting family. Then there exists a vertex~$i\in [n]$ such that the link of~$i$ satisfies~$\vert\partial\cF(i)\vert\geq\vert\cF(i)\vert$. We prove this conjecture for~$n>\binom{k+1}{2}$.
		\begin{theorem}\label{thm:main}
			Let~$n> \binom{k+1}{2}$ and let~$\cF\subseteq[n]^{(k)}$ be intersecting.
			Then there exists an~$i\in[n]$ such that~$\vert\partial\cF(i)\vert\geq\vert\cF(i)\vert$.
		\end{theorem}
	
		%Frankl further conjectured that the analogous statement is true for cross-intersecting families.
		Frankl~\cite{F:21} further conjectured that a local version of Theorem \ref{thm:Franklcrossint} should hold. We prove it for $n>k\ell$.
		
		%\begin{conj}[Frankl]\label{conj:Franklcrossint}
		%	Suppose that for~$n> k+\ell$, the families~$\cG\subseteq [n]^{(\ell)}$ and~$\cH\subseteq [n]^{(k)}$ are cross-intersecting.
		%	Then there is some~$i\in[n]$ such that $$\vert\partial\cG(i)\vert\geq\vert\cG(i)\vert\text{ or }\vert\partial\cH(i)\vert\geq\vert\cH(i)\vert\,.$$
		%\end{conj}
		
		%We prove this conjecture for~$n>k\ell$.
		
		\begin{theorem}\label{thm:crossint}
			Suppose that for~$n>k\ell$, the families~$\cF\subseteq [n]^{(k)}$ and~$\cG\subseteq [n]^{(\ell)}$ are cross-intersecting.
			Then there is some~$i\in[n]$ such that $$\vert\partial\cF(i)\vert\geq\vert\cF(i)\vert\text{ or }\vert\partial\cG(i)\vert\geq\vert\cG(i)\vert\,.$$
		\end{theorem}
		
		%In fact, we prove more general results.	
		%For a family~$\cF\subseteq[n]^{(k)}$ and sets~$A,B,C\subseteq[n]$, we write
		%\begin{align*}
		%	\cF(A,\overline{B\setminus A})=\{F\setminus A:F\in\cF\text{ and }A\subseteq F\text{ and }F\cap (B\setminus A)=\emptyset\}\,,
		%\end{align*}
		%which can be thought of as the link of~$\cF$ at~$A$ induced on~$[n]\setminus(B\setminus A)$, and
		%\begin{align*}
		%	\cF^{\cup C}=\{F\cup C:F\in\cF\}\,.
		%\end{align*}
	
		Theorems~\ref{thm:main} and~$\ref{thm:crossint}$ will be deduced from a more general result. To formulate it compactly, let us further introduce the following definition.
		\begin{dfn}
			We say that~$\cF\subseteq[n]^{(k)}$ is \textit{pseudo-intersecting} if~$\vert\partial\cF(\overline{X})\vert\geq\vert\cF(\overline{X})\vert$ for all~$X\subseteq[n]$.
		\end{dfn}

		Note that for~$A\subseteq [n]$ and~$\cF\subseteq[n]^{(k)}$, the link~$\cF(A)$ being pseudo-intersecting means that~$\vert\partial\cF(A,\overline{X\setminus A})\vert\geq\vert\cF(A,\overline{X\setminus A})\vert$ for every~$X\subseteq [n]$. 
		The term pseudo-intersecting comes from the following observation: If~$\cF\subseteq [n]^{(k)}$ is intersecting, then for every~$X\subseteq [n]$ we have~$\vert\partial\cF(\overline{X})\vert\geq\vert\cF(\overline{X})\vert$.
		This is a consequence of Theorem~\ref{thm:Katona} and the fact that~$\cF(\overline{X})\subseteq \cF$ is intersecting.
		Note that if~$\cF(A)$ is pseudo-intersecting for~$A\in [n]^{(j)}$ with $j<k$, then~$\vert\partial\cF(A)\vert\geq\vert\cF(A)\vert$. That is, the pseudo-intersecting property of~$\cF(i)$ implies the local version of Katona's intersecting theorem.
		
		%Note that if for some family~$\cF\subseteq [n]^{(k)}$ and~$i\in[n]$, the family~$\cF(i)$ is pseudo-intersecting, then in particular~$\vert\partial\cF(i)\vert\geq\vert\cF(i)\vert$.
		%More generally, let~$A,M\subseteq[n]$.
		%Then~$\cF(A,\overline{M\setminus A})$ is pseudo-intersecting if and only if for all~$X\subseteq[n]$ with~$M\subseteq X$, we have
		%\begin{align*}
		%	\vert\partial\cF(A,\overline{X\setminus A})\vert\geq\vert\cF(A,\overline{X\setminus A})\vert\,.
		%\end{align*}
		%The following result implies Theorem~\ref{thm:main} for~$n>\frac{(k+1)k}{2}$.

		The next result shows that one can always find a pseudo-intersecting link in an intersecting family.
		In particular, it implies Theorem~\ref{thm:main} for~$n>\binom{k+1}{2}$.
		
		\begin{theorem}\label{thm:general}
			Let~$\cF\subseteq[n]^{(k)}$ be intersecting.
			Then there are sets~$M_{1}\subseteq \dots\subseteq M_{k}\subseteq[n]$ with~$\vert M_i\vert\geq n-\sum_{i\leq j\leq k}j$ such that for all~$A\in M_i^{(i)}$, the family~$\cF(A)$ is pseudo-intersecting.
		\end{theorem}
		
		Note that the link of every subset of~$M_1$ is pseudo-intersecting.
		In particular, for any intersecting family~$\cF\subseteq [n]^{(k)}$, the inequality~$\vert\partial\cF(i)\vert\geq\vert\cF(i)\vert$ holds for all but at most~$\binom{k+1}{2}$ vertices~$i$.
		
		Our general result for cross-intersecting families reads as follows.
		For~$n>k\ell$ it implies Theorem~\ref{thm:crossint}.
		
		\begin{theorem}\label{thm:crossintgeneral}
			Let~$\cF\subseteq [n]^{(k)}$ and~$\cG\subseteq [n]^{(\ell)}$ be cross-intersecting.
			Then there are sets~$M_1\subseteq\dots\subseteq M_{k}\subseteq[n]$ with~$\vert M_i\vert\geq n-(k+1-i)\ell$ such that one of the following holds.
			\begin{itemize}
				\item The familiy~$\cF(A)$ is pseudo-intersecting for all~$i\in[k]$ and~$A\subseteq M_i^{(i)}$,
				\item or~$\cG(B)$ is pseudo-intersecting for all~$B\subseteq M_2$.
			\end{itemize}
		\end{theorem}
		
		A family~$\cF\subseteq \cP([n])$ is called \textit{$t$-union} if~$\vert F\cup F' \vert \leq t$ for every~$F, F' \in \cF$. A family~$\cF$ is an \textit{antichain} if~$F\not \subseteq F'$ for every $F, F'\in \cF$. Kiselev, Kupavskii and Patk\'{o}s made the following conjecture on the minimum degree of~$(2\ell+1)$-union antichains.\footnote{They formulated their conjecture as an upper bound on the diversity of an intersecting antichain, which is equivalent.}
		%Frankl made the conjectures above in his work on the following conjecture by Kiselev, Kupavskii, and Patk\'os~\cite{F:21} on an upper bound for the minimum degree of~$(2\ell+1)$-union antichains.\footnote{They formulated their conjecture as an upper bound on the diversity of an intersecting antichain, which is equivalent.}
		
		\begin{conj}[\cite{F:21}]\label{conj:Kupavskii}
			Suppose that~$1\leq 2\ell+1<n$, and that~$\cF\subseteq\ccP([n])$ is a~$(2\ell+1)$-union antichain.
			Then~$\delta(\cF)\leq\binom{n-1}{\ell-1}$.
		\end{conj}
		
		In~\cite{F:21}, Frankl solved Conjecture~\ref{conj:Kupavskii} for~$n\geq\ell^3+\ell^2+\frac{3}{2}\ell$. He also noted that one could obtain better bounds by proving a local version of Theorem~\ref{thm:Katona} (see Proposition 3.4(i), \cite{F:21}). In particular, by using his reduction, Theorem~\ref{thm:main} implies Conjecture~\ref{conj:Kupavskii} for $n>\binom{\ell+2}{2}$.
		
		%and reduced it to Conjecture~\ref{conj:Frankl} (he further solved the cases~$\ell=1,2$).
		%Thus, using his reduction, Theorem~\ref{thm:main} in particular implies that Conjecture~\ref{conj:Kupavskii} is true for~$n> \frac{(\ell+2)(\ell+1)}{2}$.

	\section{Tools}\label{sec:tools}
		%Frankl~\cite{F:76} showed that Katona's intersection theorem is also true for cross-intersecting families.
		%\begin{theorem}[Frankl]\label{thm:Franklcrossint}
		%	Let~$\cF\subseteq[n]^{(k)}$ and~$\cG\subseteq [n]^{(\ell)}$ be cross-intersecting.
		%	Then for some~$\cH\in\{\cF,\cG\}$, we have~$\vert\partial\cH\vert\geq\vert\cH\vert$.
		%\end{theorem}
	
		In this section we introduce the main technical lemma in the proof. Roughly speaking, in our proof we will inductively construct sets~$M_i\subseteq M_{i+1}$ such that~$\cF(A)$ is pseudo-intersecting for all~$A\subseteq M_i^{(i)}$.
		The following lemma will help with the induction step.
		We directly formulate it in the setup in which it will be used.
		
		\begin{lemma}\label{lem:shadow}
			Let~$\cF\subseteq[n]^{(k)}$ and let~$A,M\subseteq[n]$.
			If~$\cF(A,\overline{M\setminus A})$ is pseudo-intersecting and~$\cF(A\cup x)$ is pseudo-intersecting for all~$x\in M\setminus A$, then~$\cF(A)$ is pseudo-intersecting.
		\end{lemma}
		
		\begin{proof}
			Let~$X\subseteq[n]$ be given.
			We need to show that~$\vert\partial\cF(A,\overline{X\setminus A})\vert\geq\vert\cF(A,\overline{X\setminus A})\vert$.
			Note that~$\vert\partial\cF(A,\overline{(X\cup M)\setminus A})\vert\geq\vert\cF(A,\overline{(X\cup M)\setminus A})\vert$ holds by assumption.
			Hence, it is enough to show that if
			\begin{align}\label{eq:X'hypo}
				\vert\partial\cF(A,\overline{X'\setminus A})\vert\geq\vert\cF(A,\overline{X'\setminus A})\vert
			\end{align}
			holds for some~$X\subsetneq X'\subseteq X\cup M$, then~$\vert\partial\cF(A,\overline{(X'\setminus x)\setminus A})\vert\geq\vert\cF(A,\overline{(X'\setminus x)\setminus A})\vert$ for some~$x\in X'\setminus (X\cup A)$.
			So let~$X'$ with~$X\subsetneq X'\subseteq X\cup M$ satisfy~\eqref{eq:X'hypo} and let~$x\in X'\setminus (X\cup A)$ be arbitrary. Given a family~$\cH$ and a vertex~$x$, let~$\cH^{\cup x}=\{H\cup\{x\}:\: H\in \cH\}$.
			Observe that
			\begin{align}\label{eq:shrinkfam1}
				\cF(A,\overline{(X'\setminus x)\setminus A})=\cF(A,\overline{X'\setminus A})\dcup(\cF(A\cup x,\overline{X'\setminus (A\cup x)}))^{\cup x}\,
			\end{align}
			and
			\begin{align}\label{eq:shrinkfam2}
				\partial\cF(A,\overline{(X'\setminus x)\setminus A})\supseteq\partial\cF(A,\overline{X'\setminus A})\dcup(\partial\cF(A\cup x,\overline{X'\setminus (A\cup x)}))^{\cup x}\,.
			\end{align}
			Since we have~\eqref{eq:X'hypo} and since
			\begin{align*}
				\vert\partial\cF(A\cup x,\overline{X'\setminus (A\cup x)})\vert\geq\vert\cF(A\cup x,\overline{X'\setminus (A\cup x)})\vert
			\end{align*}	
			holds because~$\cF(A\cup x)$ is pseudo-intersecting, \eqref{eq:shrinkfam1} and~\eqref{eq:shrinkfam2} imply that
			\begin{align*}
				\vert\partial\cF(A,\overline{(X'\setminus x)\setminus A})\vert\geq\vert\cF(A,\overline{(X'\setminus x) \setminus A})\vert\,.
			\end{align*}
			This is all we had to show.
		\end{proof}
	
		\section{Proof of Theorems~\ref{thm:general} and~\ref{thm:crossintgeneral}}
		
		\begin{proof}[Proof of Theorem~\ref{thm:general}]
			The proof proceeds inductively by constructing sets~$M_i\subseteq M_{i+1}$ such that for all~$A\in M_{i}^{(i)}$, the family~$\cF(A)$ is pseudo-intersecting and~$\vert M_i\vert\geq n-\sum_{i\leq j\leq k}j$.
			
			We begin the backward induction with~$i=k$.
			If~$\cF\neq\emptyset$, we would be done, so let~$F\in\cF$ and set~$M_k=[n]\setminus F$ (in particular,~$\vert M_k\vert\geq n-k$).
			Since~$\cF$ is intersecting, we have~$\cF\cap M_k^{(k)}=\emptyset$.
			Thus, for~$A\in M_k^{(k)}$ and~$X\subseteq[n]$ we have~$\cF(A,\overline{X\setminus A})=\emptyset$. Hence, $\vert\partial\cF(A,\overline{X\setminus A})\vert\geq\vert\cF(A,\overline{X\setminus A})\vert$ for all~$X\subseteq [n]$, meaning that~$\cF(A)$ is pseudo-intersecting for~$A\in M_k^{(k)}$.
			
			Now assume that for some~$i$ with~$2\leq i\leq k$, a set~$M_i\subseteq[n]$ with~$\vert M_i\vert\geq n-\sum_{i\leq j\leq k}j$ has been defined such that~$\cF(A^+)$ is pseudo-intersecting for all~$A^+\in M_i^{(i)}$.
			Next, we will argue that we only need to delete at most one~$(i-1)$-set from~$M_i$ to obtain a set~$M_{i-1}$ as desired.
			
			If for all~$A\in M_i^{(i-1)}$ the family~$\cF(A,\overline{M_i\setminus A})$ is pseudo-intersecting, then we set~$M_{i-1}=M_i$.
			Since the induction hypothesis tells us that for every~$A\in M_i^{(i-1)}$ and~$x\in M_i\setminus A$, the family~$\cF(A\cup x)$ is pseudo-intersecting, Lemma~\ref{lem:shadow} implies that~$\cF(A)$ is pseudo-intersecting.
			So let us assume that~$\cF(B,\overline{M_i\setminus B})$ is not pseudo-intersecting for some~$B\in M_i^{(i-1)}$.
			Then there is some~$M\subseteq[n]$ with~$M_i\subseteq M$ such that~$\vert\partial\cF(B,\overline{M\setminus B})\vert <\vert\cF(B,\overline{M\setminus B})\vert$ and we set~$M_{i-1}=M_i\setminus B$.
			\begin{claim}\label{cl:prepseudoint}
				For every~$A\in M_{i-1}^{(i-1)}$, the family~$\cF(A,\overline{M_i\setminus A})$ is pseudo-intersecting.
			\end{claim}
			\begin{proof}
				Let~$A\in M_{i-1}^{(i-1)}$,~$M'\subseteq[n]$ with~$M_i\subseteq M'$ and consider~$F_1\in\cF(A,\overline{M'\setminus A})$ and~$F_2\in\cF(B,\overline{M\setminus B})$.
				Then~$F_1\cup A,F_2\cup B\in\cF$ and~$F_1,F_2\subseteq [n]\setminus M_i\subseteq [n]\setminus(A\cup B)$.
				Since~$A$ and~$B$ are disjoint and~$\cF$ is intersecting,~$F_1\cap F_2\neq\emptyset$ and so~$\cF(A,\overline{M'\setminus A})$ and~$\cF(B,\overline{M\setminus B})$ are cross-intersecting.
				Theorem~\ref{thm:Franklcrossint} implies that~$\vert\partial\cH\vert\geq\vert\cH\vert$ has to hold for some~$\cH\in\{\cF(A,\overline{M'\setminus A}),\cF(B,\overline{M\setminus B})\}$.
				By the choice of~$B$ and~$M$, this yields the statement of the claim.
			\end{proof}
			Together with the induction hypothesis, this claim allows us to apply Lemma~\ref{lem:shadow} which yields that~$\cF(A)$ is pseudo-intersecting for all~$A\in M_{i-1}^{(i-1)}$.
			Further note that in either case~$\vert M_{i-1}\vert\geq\vert M_i\vert -i+1\geq n-\sum_{i-1\leq j\leq k}j$.
			Therefore, in either case~$M_{i-1}$ is as desired.
		\end{proof}
	
		\begin{proof}[Proof of Theorem~\ref{thm:crossintgeneral}]
			Again, we aim to inductively construct sets~$M_{i}\subseteq M_{i+1}$ such that for all~$A\in M_{i}^{(i)}$, the family~$\cF(A)$ is pseudo-intersecting and~$\vert M_i\vert\geq n-(k-i+1)\ell$.
			If at any point, we should not be able to proceed, i.e., we fail to construct the set~$M_{i-1}$, then~$\cG(A)$ will be pseudo-intersecting for all~$A\subseteq M_{i}$.
			
			We begin the backwards induction with~$i=k$.
			If~$\cG=\emptyset$, we are done, so let~$G\in \cG$ and set~$M_k=[n]\setminus G$.
			Since~$\cF$ and~$\cG$ are cross-intersecting, this means that~$\cF\cap M_k^{(k)}=\emptyset$ and therefore,~$\cF(A)$ is pseudo-intersecting for all~$A\in M_k^{(k)}$.
			Further, we have~$\vert M_k\vert\geq n-\ell$.
			
			Now assume that for some~$i$ with~$2\leq i\leq k$ we have constructed sets~$M_j$ for all~$i\leq j\leq k$ as desired.
			First assume that~$\cG(B,\overline{M_i\setminus B})$ is not pseudo-intersecting for some~$B\in M_i^{(\leq \ell)}$, i.e., there is some~$X\subseteq[n]$ with~$M_i\subseteq X$ such that~$\vert\partial\cG(B,\overline{X\setminus B})\vert<\vert\cG(B,\overline{X\setminus B})\vert$.
			Then we set~$M_{i-1}=M_i\setminus B$ and readily notice that~$\vert M_{i-1}\vert\geq n-(k-i+2)\ell$.
			Further, observe that (similarly as in the proof of Claim~\ref{cl:prepseudoint}) for all~$A\in M_{i-1}^{(i-1)}$ and $M_i\subseteq M'\subseteq [n]$, the families~$\cF(A,\overline{M'\setminus A})$ and~$\cG(B,\overline{X\setminus B})$ are cross-intersecting since~$A\cap B=\emptyset$,~$A\cup B\subseteq M_i\subseteq M',X$, and since $\cF$ and $\cG$ are cross-intersecting.
			Thus, by the choice of~$B$ and~$X$, Theorem~\ref{thm:Franklcrossint} yields that~$\cF(A,\overline{M_i\setminus A})$ is pseudo-intersecting.
			Since the induction gives us that~$\cF(A^+)$ is pseudo-intersecting for all~$A^+\in M_i^{(i)}$, we are now in a position to apply Lemma~\ref{lem:shadow} to conclude that~$\cF_1(A)$ is pseudo-intersecting for all~$A\in M_{i-1}^{(i-1)}$.
			
			Next assume that~$\cG(B,\overline{M_i\setminus B})$ is pseudo-intersecting for all~$B\in M_i^{(\leq \ell)}$.
			In this case we set~$M_1=M_2=\dots=M_i$.
			\begin{claim}
				The family~$\cG(B)$ is pseudo-intersecting for all~$B\subseteq M_i=M_2$.
			\end{claim}
			\begin{proof}
				For~$B\subseteq M_i$ with~$\vert B\vert>\ell$, the statement follows immediately.
				We proceed by backwards induction on~$j=\vert B\vert$ and begin with~$j=\ell$.
 				The induction start follows because for~$B$ of size~$\ell$, we have $\cG(B,\overline{X\setminus B})=\cG(B,\overline{M_i\setminus B})$ for all~$X\subseteq[n]$ and~$\cG(B,\overline{M_i\setminus B})$ is pseudo-intersecting.
				
				Given that for some~$j\leq \ell$, the family~$\cG(B^+)$ is pseudo-intersecting for all~$B^+\in M_i^{(j)}$, we can apply Lemma~\ref{lem:shadow} (since we are in the case that~$\cG(B,\overline{M_i\setminus B})$ is pseudo-intersecting for all~$B\in M_i^{(\leq \ell)}$) to conclude that~$\cG(B)$ is pseudo-intersecting for all~$B\in M_i^{(j-1)}$ which finishes the induction step.
			\end{proof}
		
			Thus, we have shown that if we cannot construct all sets~$M_1,\dots,M_k$ as desired by induction, then there is a set~$M_2\subseteq[n]$ with~$\vert M_2\vert\geq n-(k-1)\ell$ such that~$\cG(B)$ is pseudo-intersecting for all~$B\subseteq M_2$.
			In other words, we proved that indeed one of the statements in the theorem has to hold.
		\end{proof}
		
		\section*{Acknowledgments}
			The authors thank Alexandre Perozim de Faveri for fruitful discussions and Peter Frankl and Andrey Kupavskii for reading earlier versions of this paper.

\end{document}